\documentclass[11pt]{amsart}  
\usepackage{bm}
\usepackage{amscd}
\usepackage{amsmath}
\usepackage{amssymb}
\usepackage{amsthm}
\usepackage{epsf}
\usepackage{verbatim}
\input epsf.tex
\usepackage{graphicx}
\usepackage{mathtools}
\usepackage{enumerate}
\usepackage{color}
\newtheorem{theorem}{Theorem}
\newtheorem{lemma}[theorem]{Lemma}

\newtheorem{remark}[theorem]{Remark}

\newtheorem{claim}[theorem]{Claim}
\def\hp{\widehat{p}}
\def\hps{\widehat{\psi}}

\newcommand{\W}{\mathcal{W}}

\newcommand{\ve}{\varepsilon}

\newcommand{\lo}{\left( 1+ o(1) \right)}

\newcommand{\beq}[2]{\begin{equation}\label{#1}#2\end{equation}}

\DeclarePairedDelimiter{\defaultDelim}{[}{]}

\DeclareMathOperator{\capPr}{\sf Pr}
\renewcommand{\Pr}[2][]{\capPr_{#1}\defaultDelim*{#2}}
\DeclareMathOperator{\capE}{\sf E}
\newcommand{\E}[2][]{\capE_{#1}\defaultDelim*{#2}}

\newcommand{\set}[1]{\left\{#1\right\}}

\def\hW{\widehat{\mathcal W}}
\def\a{\alpha}
\def\b{\beta}
\def\d{\delta}

\def\e{\varepsilon}

\def\G{\Gamma}

\def\th{\theta}
\def\Th{\Theta}
\def\l{\lambda}

\def\n{\nu}
\def\p{\pi}

\def\r{\rho}

\def\t{\tau}
\def\om{\omega}
\def\Om{\Omega}

\def\cW{{\mathcal W}}
\def\La{\Lambda}

\newcommand{\brac}[1]{\left(#1\right)}
\newcommand{\bfrac}[2]{\left(\frac{#1}{#2}\right)}
\begin{document}
\title{The covertime of a biased random walk on $G_{n,p}$}
\author{ Colin Cooper}\thanks{Department of Computer Science, Kings College, London WC2R 2LS. Research supported in part by EPSRC grant EP/M005038/1}
\author{Alan Frieze}\thanks{Department of Mathematical Sciences, Carnegie Mellon University, Pittsburgh PA15213. Research supported in part by NSF grant DMS0753472}
\author{Samantha Petti}\thanks{School of Mathematics, Georgia Tech., Atlanta, GA30313. This material is based upon work supported by the National Science Foundation Graduate Research Fellowship under Grant No. DGE-1650044.} 
\begin{abstract}
We analyze the covertime of a biased random walk on the random graph $G_{n,p}$. The walk is biased towrds visiting vertices of low degree and this makes the covertime less than in the unbiased case.
\end{abstract}
\maketitle
\section{Introduction}
Let $G=(V,E)$ be a connected graph with $n$ vertices and $m$ edges.
For $v\in V$, let $C_v$ be the expected time for a simple random
walk $\cW_v$ on $G$ starting at $v$, to visit every vertex of $G$. The
{\em vertex cover time} $C(G)$ of $G$ is defined as $C(G)=\max_{v\in
V}C_v$. The vertex cover time of connected graphs has been
extensively studied. It is a classic result of Aleliunas, Karp,
Lipton, Lov\'asz and Rackoff \cite{AKLLR} that $C(G) \le 2m(n-1)$. It
was shown by
 Feige \cite{Feige}, \cite{Feige1},
that for any connected graph $G$, the cover time satisfies
$(1-o(1))n\log n\leq C(G)\leq (1+o(1))\frac{4}{27}n^3.$
The asymptotic lower bound  is obtained by the complete graph $K_n$.
The asymptotic upper bound  is obtained by the {\em lollipop} graph, which consists of a path of length $n/3$ joined to a
clique of size $2n/3$.

The facts that the cover time of a simple random walk can be as large as $\Th(n^3)$ for some classes of graphs, and that it is never $o(n \log n)$ for any graph encourages a study of modified random walks whose performance may be better (in order of magnitude), either in general or for specific classes of graphs. If we restrict our attention to reversible random walks then the lower bound cannot be improved.

The properties of weighted random walk on an undirected graph are as follows, for more details see  \cite{AlFi}. Let $p(u,v)$ denote the probability of moving from vertex $u$ to vertex $v$ in a single step and let $\p(v)$ be the steady state distribution for the walk, assuming that it exists. A Markov chain is {\em reversible} if $\pi(u)p(u,v)=\pi(v) p(v,u)$. Suppose now that each undirected edge $e=\{u,v\}$ has a positive weight  $w_e=w_{v,u}=w_{u,v}$. The  transition probability of the associated random walk  at $v$ is $p(v,u)=w_{v,u}/w_v$, where $w_v=\sum_{u \in N(v)} w_{v,u}$, and $N(v)$ are the neighbors of $v$ in $G$. The stationary distribution of the walk at vertex $v$ is $\pi(v)=w_v/w$, where $w=\sum_{e \in E} w_e$. We can verify this and the fact that  weighted walks are always reversible since for any edge $e=\{u,v\}$, $\pi(u)p(u,v)=w_e/w=\pi(v)p(v,u)$.

As an example, for a simple random walk, we take $w_{u,v}=1$, $w_v=d(v)$ the degree of vertex $v \in V$, and $p(u,v)=1/d(v)$; the total weight $w=2m$, and so $\pi(v)=d(v)/2m$. An $O(n^2 \log n)$ upper bound on the cover time for any connected $n$ vertex graph $G$, was obtained by Ikeda, Kubo, Okumoto, and Yamashita \cite{IKOY} by using a weight $w_{u,v}=1/\sqrt{d(u)d(v)}$
for edge $e=\{u,v\}$. The fact that the above edge weight is multiplicative makes the walk hard to analyze. The use of a simpler but related weight of $w(u,v)=1/\min(d(u),  d(v))=\max(1/d(u),1/d(v))$ was studied in \cite{ACD}. Using this simplified weight David and Feige \cite{DFeige2017} proved an $O(n^2)$ upper bound on cover time for {\em any} connected $n$ vertex graph $G$. As the cover time of  paths and cycles by weighted walks is $\Th(n^2)$ this result is best possible. Instead of choosing a uniform random neighbor, the walks of  \cite{ACD},\cite{IKOY}  are biased towards lower degree vertices. In this way the walk tends to have a smaller cover time than the unbiased walk.

In this paper we study the cover time of the walk in \cite{ACD}, \cite{DFeige2017}. We will analyze its performance on the random graph $G_{n,p}$. The walk is biased towards lower degree vertices. In this way the walk will tend to have a smaller covertime than the unbiased walk. We use the following notation concerning a graph $G=(V,E)$: 
\begin{enumerate}
\item $d(v)=d_G(v)$ is the degree of $v\in V$.
\item We let $\psi(v,w)=\frac{1}{\min\set{d(v),d(w)}}$ for $\set{v,w}\in E$.
\item We let $\Psi(v)=\sum_{w\in N(v)}\psi(v,w)$.
\end{enumerate} 
The random walk $\cW_u=(X_0=u\in V,X_1,\ldots,X_t,\ldots)$ is then defined by
\beq{walk}{
\Pr{X_{t+1}=w\mid X_t=v}=\begin{cases}\frac{\psi(v,w)}{\Psi(v)}&w\in N(v)\\0&w\notin N(v)\end{cases}.
}
Let $C_u(G)$ be the expected time for $\cW_u$ to visit all of $V$ and let $\overline{C}(G)=\max_uC_u(G)$ denote the cover time of this random walk. We will prove the following theorem.
\begin{theorem}\label{main} Let $G \sim G_{n,p}$ where $p=\frac{c\log{n}}{n}$ and where $\om=(c-1)\log n\to\infty$. Then with high probability, $\overline{C}(G)\approx n \log{n}$.
\end{theorem}

In Section \ref{fv lemma} we state the central lemma for the proof of Theorem \ref{main}. The ``first visit time lemma" bounds the probability that a vertex has not been visited in $t$ steps after a suitably defined mixing time. For a proof of this lemma, in the stated form, see \cite{CF}. In Section \ref{typical section}, we describe relevant properties of $G_{n,p}$ that hold with high probability, and compute quantities necessary for applying the first visit lemma under these conditions. In Section \ref{the proof} we prove Theorem \ref{main}.

For probabilistic inequalities we use the Chernoff bounds on the binomial $Bin(n,p)$:
\begin{align}
\Pr{Bin(n,p)\leq (1-\e)np}&\leq e^{-\e^2np/2}\text{ for }0\leq\e\leq1.\label{Chern1}\\
\Pr{Bin(n,p)\geq (1+\e)np}&\leq e^{-\e^2np/3}\text{ for }0\leq\e\leq1.\label{Chern2}\\
\Pr{Bin(n,p)\geq \a np}&\leq \bfrac{e}{\a}^{\a np}\text{ for }\a>0.\label{Chern3}
\end{align}
We sometimes write $A_n\approx B_n$ (resp. $A_n\lesssim B_n$) in place of $A_n=(1+o(1))B_n$ (resp. $A_n\leq(1+o(1))B_n$) as $n\to\infty$.

{\bf Some further notation:}
\begin{enumerate}
\item For $S\subseteq V$ we let $N(S)=N_G(S)=\set{w\notin S:\set{v,w}\in E}$ be the disjoint neighborhood of $S$. Let $d(S)= \sum_{v \in S} d(v)$. 
\item We abbreviate $N(\set{v})$ to $N(v)=N_G(v)$ for $v\in V$. Thus, $d(v)=|N(v)|$.
\item For $v\in V$ and positive integer $k$ we let $N_k(v)$ denote the set of vertices within distance at most $k$ of $v$.
\end{enumerate}
\section{The first visit time lemma}\label{fv lemma}
Let $G=(V,E)$ be a fixed graph, and let $u\in V$ be arbitrary. Let $\W_u$ denote the modified random walk defined in \eqref{walk} starting with $X_0=u$. The walk defines a reversible Markov chain with state space $V$. Let $P$ be the matrix of transition probabilities, and $\pi_v$ the stationary distribution of $P$. As previously mentioned,
 for $v\in V$,
\beq{pi}{
\p_v=\frac{\Psi(v)}{\sum_{u\in V}\Psi(u)}.
}
Considering the  walk $\cW_v$, starting
at $v$, let $r_t=\Pr{\cW_v(t)=v}$ be the probability  that this  walk
returns to $v$ at step $t \ge 0$, and let
$$R(z)=\sum_{t=0}^\infty r_tz^t$$
generate $r_t$. Our definition of return includes $r_0=1$.
For $R(z)$ and given $T$
 let
\[
R(T, z)=\sum_{j=0}^{T-1} r_jz^j.
\]
We choose a value of $T$ given by
\begin{equation}\label{T}
T=  L \log n,
\end{equation}
 and for this value of $T$
let $R_v=R(T,1)$.
In Lemma \ref{mixing time claim} we put a lower bound on the constant $L$ which is sufficient to imply that $T$ is the mixing time of $\W_u$, in a well-defined sense.

The following first visit time lemma bounds the probability a vertex has not been visited in time $T, T+1, \dots t$.
\begin{lemma} \label{firstvisit}[The first visit time lemma \cite{CF}]\\
Let $G$ be a graph satisfying the following conditions
\begin{enumerate}[(i)]
\item \label{mixing time} For all $t\geq T$, $\max_{u,x \in V} |P_u^{(t)}(x)- \pi_x|\leq n^{-3}$
\item \label{rt} For some (small) constant $\theta>0$ and some (large) constant $K >0$, $$ \min_{|z|\leq 1+ \frac{1}{KT}} |R(T,z)| \geq \theta$$
\item $T\pi_v=o(1)$ and $T\pi_v = \Omega(n^{-2})$
\end{enumerate}
Let $\mathcal{A}_v(t)$ be the event that the random walk $\mathcal{W}_u$ on graph $G$ does not visit vertex $v$ at steps $T, T+1, \dots t$. Then, uniformly in $v$,
$$\Pr{\mathcal{A}_v(t)} = \frac{(1+O(T\pi_v))}{(1+p_v)^t} +O(T^2 \pi_v e^{-t/KT})$$
where $p_v$ is given by the following formula, with $R_v= R_v(T,1)\geq 1$:
$$p_v= \frac{\pi_v}{R_v(1+O(T\pi_v))}.$$
In our applications, $T=O(\log n),\p_v=O(1/n)$ and $t=\Theta(n\log n)$. In which case we can write
\beq{PrAt}{
\Pr{\mathcal{A}_v(t)} \approx e^{-t\p_v/R_v}.
}
\end{lemma}

 We rely on the following lemma to show Condition ($\ref{rt}$) of the first visit time lemma.

\begin{lemma}\label{helper}
Lwmma 18 of \cite{Hyp} proves the following: Let $v$ be a vertex of an arbitrary graph $G$. Let $T$ be a mixing time satisfying (\ref{mixing time}). If $T= o(n^3)$, $T\pi_v=o(1)$ and $R_v$ is bounded above by a constant, then Condition (\ref{rt}) of Lemma \ref{firstvisit} hold for $\theta= 1/4$ and any constant $K \geq 3R_v$.
\end{lemma}

\section{Properties of typical graphs $G \sim G_{n,p},p=\frac{c\log n}{n}$}\label{typical section}
The following lemma defines a \textit{typical} graph and shows that with high probability a graph $G \sim G_{n,p},p=\frac{c\log n}{n}$ is typical. In Lemma \ref{mixing time claim} we show that $T$ as given in Equation \eqref{T} is the mixing time for a typical graph, and in Lemma \ref{rv}, we bound $R_v$ for a typical graph.  
\begin{lemma} \label{typical} Let $\ve > 0$ be an arbitrary small constant. Consider $G \sim G_{n,p},p=\frac{c\log n}{n}$ with $c\geq1$ and $\om=(c-1)\log n\to\infty$. The graph $G$ is ``typical" if all of the following conditions are satisfied:
 \begin{enumerate}[(a)]
\item \label{con} $G$ is connected.
\item \label{t1} There are at most $n^{1-\ve^2c/4}$ vertices with degree less than $ (1-\ve) np$.
\item \label{t2} There are at most $n^{1-\ve^2c/4}$ vertices with degree more than $(1+\ve) np$.
\item \label{maxd} No vertex has degree more than $4np$.
\item \label{Vk} Let $V_k$ denote the vertices of degree $k$. Then $|V_k|\leq (3\log n)^{k+1}$ for $k\leq \La=\log\log n$.
\item \label{set A} Let $A$ be the set of vertices with degree less than $ np/100$. Then (i) $|A|< n^{17/12-c}$ and (ii) no vertex of $A$ is within distance $\La$ of a cycle of size less than $\La$ and (iii) for all $u , v \in A$, the distance $dist(u,v) >\La$.
\item \label{subgraphs} Suppose that $np\leq n^{1/100}$. Then $G$ contains no subgraphs $H$ with number of vertices $v_H\leq 50$ and with number of edges $e_H \geq v_H +1$. 
\item \label{big s} For all $S\subseteq V=[n]$ such that $1/p \leq |S| \leq \frac{n}{2}$, $e(S, \overline{S}) \geq \frac12|S|(n-|S|)p$ where $e(S, \overline{S}) $ is the number of edges between the set $S$ and its complement $\overline{S}$.
\item \label{small s} For all $S\subseteq V$ such that $|S| < 1/p$, $e(S, S)< |S|np/1000$, where $e(S, S) $ is the number of edges within the set $S$.
\end{enumerate}  
With high probability, $G$ is typical. 
\end{lemma}
\begin{proof}
 (\ref{con}) This is a standard result and follows from Erd\H{o}s and R\'enyi \cite{ER}.
(\ref{t1}) By the Chernoff bound \eqref{Chern1},
$$\Pr{ d(v)  < (1-\ve)np}= \Pr{Bin(n-1,p)\leq (1-\ve)(n-1)p}\leq \exp\left( -\frac{ \ve^2np}{2+o(1)}\right)\leq n^{-\ve^2c/3}.$$
Let $X$ be the number of vertices of degree less than $(1-\ve)np$. The Markov inequality implies that 
$$\Pr{ X \geq n^{1-\ve^2c/4}} < \Pr{ X \geq n^{\ve^2 c/12} \E{X}}< n^{-\ve^2 c/12} =o(1).$$
(\ref{t2}) By the Chernoff bound \eqref{Chern2}, 
$$\Pr{ d(v)  > (1+\ve)np}= \Pr{Bin(n-1,p)\geq (1+\e)np}\leq \exp\left( -\frac{ \ve^2 np}{3}\right)= n^{-\ve^2c/3}.$$  
Let $X$ be the number of vertices of degree greater than $(1+ \ve)np$. The Markov inequality implies that 
$$\Pr{ X \geq n^{1-\ve^2c/4}} < \Pr{ X \geq n^{\ve^2c/12} \E{X}}< n^{-\e^2c/12}=o(1).$$  
(\ref{maxd}) Let $X$ be the number of vertices with degree greater than $4np$.  Then by the Chernoff bound \eqref{Chern3},
$$\Pr{X>0} \leq \E{X}\leq  n\Pr{Bin(n,p)  > 4np} \leq n \bfrac{e}{4}^{4np}< n^{1-3c/2}=o(1).$$
(\ref{Vk}) We have,
\beq{Vksize}{
\E{|V_k|}=n\binom{n-1}{k}p^k(1-p)^{n-k}\leq (np)^kn^{1-c}\exp\bfrac{ck\log n}{n}.
}
If $c\leq 2$ then the RHS of \eqref{Vksize} is at most $(2\log n)^k$ the claim now follows from the Markov inequality. When $c>2$, the RHS of \eqref{Vksize} is $o(1)$ for $k\leq \La$ and so with high probability we have $V_k=\emptyset$ for $k\leq \La$.

\noindent (\ref{set A}) First we observe that for any $a \in \{1,2,3\}$ and $b \in \{0,1,2\}$, 
\begin{align} 
\Pr{ Bin(n-a, p) < \frac{np}{100}-b}&= \sum_{i=0}^{\frac{ c \log n}{100} -b}\binom{ n -a}{i} p^i \left(1-p\right)^{n-a-i}\nonumber\\
& \leq \sum_{i=0}^{\frac{np}{100}} \left( \frac{ec \log n}{i}\right)^in^{-c}\exp\bfrac{(a+i)c\log n}{n}\label{sizeA}.
\end{align}
If $c\leq 10$ then we bound the RHS of \eqref{sizeA} by $n^{1/3-c}$. If $c\geq 10$ then we bound the RHS of \eqref{sizeA} by $n^{-2c/3}$. In summary, we bound the RHS of \eqref{sizeA} by $ n^{-\min\set{c-1/3,2c/3}}$.

To show (i), we give a probabilistic upper bound on the size of $A$. Since $d(v) \sim Bin(n-1,p)$,  $\E{|A|} \leq n^{1-\min\set{c-1/3,2c/3}}$. The Markov inequality implies $A=\emptyset$ with high probability if $c\geq 2$. Suppose then that $c\leq 2$. Then,
$$\Pr{ |A| \geq n^{17/12 -c }} < \Pr{ X \geq n^{1/12} \E{|A|}}< n^{-1/12}=o(1).$$  

Now let a cycle be small if has at most $\La$ vertices. Next we show (ii), that no vertex of $A$ is part of a small cycle or within distance $\La$ of a small cycle. To show the former,  let $X_j$ be the number of cycles on $j$ vertices that contain a vertex of $A$. Observe
\begin{align*}
\Pr{X_j>0}\leq \E{X_j}
&\leq n^j p^{j}\Pr{ Bin(n-3, p) < \frac{np}{100}-2}\\
&\leq\left(c \log n\right)^j n^{1/3-c}=o(1),
\end{align*} 
for $j\leq \La$ and $c\leq 2$.

Next, let $X_{j, \ell}$ be the number of structures that contain a $k$-cycle with path of length $\ell$ to a vertex in $A$. Observe
\begin{align*}
\Pr{X_{j, \ell}>0}\leq \E{X_{j, \ell}}
&\leq n^{j+\ell} \left( \frac{c\log n}{n}\right)^{j+\ell}\Pr{ Bin(n-2, p) < \frac{ c \log n}{100}-1}\\
&\leq \left(c \log n\right)^{j+\ell} n^{1/3-c}=o(1),
\end{align*} 
for $j,\ell\leq \La$ and $c\leq 2$.
 
Finally we show (iii), that no two vertices of $A$ are within distance $\La$ of each other.  Let $P$ be number of paths of length at most $\La$ with both ends in $A$. 
\begin{align*} 
\Pr{P>0} \leq \E{P}&\leq  n^2 \left( \Pr{ Bin(n-2,p) < \frac{c \log n}{100}-1}  \right)^2 \sum_{i=0}^\La n^{i} \left( \frac{c \log n }{n}\right)^{i+1}\\
 &\leq 2(c\log n)^{\La+1} n^{5/3-2c}=o(1),
\end{align*}
for $c\leq 2$.
(\ref{subgraphs}) Let now $X_j$ be the number of subgraphs $H$ with $v_H=j\leq 50$ and $e_H \geq v_H +1$ in $G$. Let $\n_j=O(2^{j^2})$ be the number of graphs $H$ on vertex set $[j]$ vertices with  $e_H \geq v_H +1$. Then we have
\begin{align*}
\Pr{X_j>0}\leq \E{X_j}\leq \n_j n^jp^{j+1}\leq \n_jn^{1/2}p =o(1).
\end{align*}
(\ref{big s}) Let $s= |S|$. We apply the Chernoff bound \eqref{Chern1} to obtain
$$\Pr{ e(S, \overline{S}) < \tfrac12s(n-s)p} \leq \exp\left( -\frac{s(n-s)p}{8}\right)= n^{-s(n-s)c/8n}.$$ 
Let now $X_s$ be the number of subsets $S$ of size $s\in[1/p,n/2]$  for which $e(S, \overline{S}) < \frac12s(n-s)p$. 
$$\Pr{X_s>0} \leq \E{X_s}\leq \binom{ n}{s} n^{s(n-s)c/8n}
 \leq \left( \frac{ne}{s}\cdot n^{-c/32}\right)^s \leq \left(e c\log n\cdot  n^{-c/32}\right)^s= o(n^{-1}).$$
It follows that with high probability $X_s=0$ for all $1/p\leq s\leq n/2$.
(\ref{small s}) Let $s= |S|$, and let now $X_s$ be the number of sets of size $s\leq 1/p$ with $e(S,S) \geq snp/1000$. Then,
\begin{align} 
 \Pr{X_s >0} &\leq \E{X_s} \leq \binom{n}{s} \binom{\binom{s}{2}}{snp/1000} p^{snp/1000}\nonumber\\
  &\leq \left( \frac{ne}{s}\right)^s \left( \frac{ s e }{ 2n/1000}\right)^{snp/1000} \nonumber\\
 & =\bfrac{s}{n}^{(-1+np/1000)s}\left(500e^{1+np/100}\right)^s.\label{summ}
\end{align}
Summing the RHS of \eqref{summ} for $1\leq s\leq 1/p$ completes the proof.
\end{proof}
The following claim describes the stationary distribution $\pi_v$ of the modified walk $\cW$ on a typical graph $G$. 
\begin{claim} \label{piv} 
Let $G$ be a typical graph and let $\ve >0$ be an arbitrarily small constant. Let $U$ be the set of vertices such that degree of $u$ and the degrees of all its neighbors are in the range $\left((1-\ve) c\log n,(1+\ve) c\log n\right)$. Then $|U|\approx n$ and for $u \in U$, 
$$\frac{1-\ve}{(1+\ve)n} \lesssim \pi_u \lesssim \frac{1+\ve}{(1-\ve)n}. $$ 
Moreover, for $b_1=\frac{1-\ve}{1+\ve}$  and $b_2=\frac{401}{1+\ve}$,  for all $v \in V$, 
\beq{b1b2}{
\frac{ b_1}{ n} \lesssim \pi_v \lesssim \frac{b_2}{n}.
}  
\end{claim}
\begin{proof} 
We refer the reader to \eqref{pi} for the value of $\p_v,v\in V$. We first observe that 
$$\Psi(v)\geq \sum_{w\in N(v)}\frac{1}{d(v)}=1\text{ for all }v\in V.$$
Since each vertex has degree at most $4c \log n$ and has at most one neighbor of degree less than $c \log n /100$, 
\beq{Psiv}{
\Psi(v)\in [1, 401]\text{ for }v\in V.
}
On the other hand, for $u\in U$, $\Psi(u)\in \left[1,\frac{1+\e}{1-\e}\right].$  Lemma \ref{typical}(\ref{set A}) implies that $|V \setminus U|= o(n).$ Therefore  
$$n\leq \sum_{v \in V} \Psi(v) \lesssim \bfrac{1+\ve}{1-\ve}n,$$ 
 and the statement follows. 
\end{proof}

The following lemma shows that the mixing time of $\cW_u,u\in V$ on a typical graph is $O( \log (n))$, as stated in \eqref{T}. 

\begin{lemma}\label{mixing time claim} 
Let $G=(V,E)$ be typical, and let $u$ be an arbitrary vertex of $G$. Let $P_u^{(t)}(x)$ be the probability that $\cW_u$ is at vertex $x$ at time $t$. Then  
for $b_1, b_2$ as in Claim \ref{piv}, $L= 10(8000b_2)^2/b_1^2$, and $T=L \log n $, 
$$|P_u^{(T)}(x)- \pi_x| \leq n^{-3}.$$
\end{lemma}

\begin{proof}  
It is well known, see for example \cite{LPW} that if $\lambda_2$ denotes the second largest absolute value of an eigenvalue of $P$, $\lambda_2<1$ and
\begin{equation} \label{mix} 
|P_u^{(t)}(x)- \pi_x| \leq (\pi_x/ \pi_u)^{1/2} \lambda_2^t\leq (b_2/b_1)^{1/2} \lambda_2^t\leq 21\l^t,
\end{equation} 
where the second inequality follows from \eqref{b1b2} and the third assumes that $\e$ is sufficiently small.

To compute the size of the second largest absolute value of an eigenvalue of $P$, we apply the Cheeger inequality. Recall the definition of conductance:
$$\Phi(G)= \min_{S \subset V(G), \pi(S)\leq \frac{1}{2}} \Phi(S), \quad \text{ where } \quad \Phi(S)= \frac{ \sum_{x \in S, y \in \overline{S}} \pi_x P_{x,y}}{ \sum_{x \in S} \pi_x}.$$ 
Applying Claim \ref{piv} we observe 
$$ \Phi(S) = \frac{ \sum_{x \in S, y \in \overline{S}} \pi_x P_{x,y}}{ \sum_{x \in S} \pi_x} > \frac{\frac{b_1}{n} \sum_{x \in S} \frac{e(\{x\}, \overline{S})}{d(x)}}{|S| \frac{b_2}{n}}= \frac{b_1}{|S| b_2} \sum_{x \in S} \frac{e(\{x\}, \overline{S})}{d(x)}.$$
Define 
$$D(S):=\sum_{x \in S} \frac{e(\{x\}, \overline{S})}{d(x)}.$$  
We give a lower bound on $D(S)$ for all subsets $S$ for which $\pi(S) \leq \frac{1}{2}$.   By Conditions \ref{t1}, \ref{t2} of Lemma  \ref{typical} and Claim \ref{piv}, a set $S$ for which  $\pi(S) \leq \frac{1}{2}$ can have cardinality at most $\frac{n+o(n)}{2} \left( \frac{1+\ve}{1-\ve}\right)^2\leq \frac{2n}{3}$, assuming $\e$ is sufficiently small. 

\noindent {\bf Case 1:} $1/p\leq |S| \leq  2n/3$.\\  
By Conditions \ref{maxd} and \ref{big s} of Lemma \ref{typical} 
\beq{Phi1}{
D(S)\geq \frac{ e(S, \overline{S})}{4np}\geq \frac{|S|(n-|S|)}{8n } \quad \text{ and so } \quad  \Phi(S)\geq \frac{b_1 (n-|S|)}{ 8n b_2} \geq \frac{b_1}{24b_2}.
}
{\bf Case 2:} $|S|<1/p$.\\
To evaluate the case when $|S| < 1/p$, we consider two subcases. Let $A$ be the set of vertices of degree less than $\frac{np}{100}$, as in Lemma \ref{typical}. 

\noindent {\bf Case 2a:} $|A\cap S| < 3|S|/4$.\\ 
By Condition \ref{small s} of Lemma \ref{typical}, we have 
$$e(S, \overline{S}) \geq d(S)- 2e(S,S)\geq \frac{|S|}{4} \frac{np}{100}- 2 |S| \frac{np}{1000}.$$
It follows that 
\beq{Phi2}{
D(S) \geq \frac{\frac{|S|}{4} \frac{np}{100}- 2 |S| \frac{np}{1000}}{ 4 c\log n} = \frac{|S|}{8000}  \quad \text{ and so } \quad  \Phi(S)  \geq \frac{b_1}{8000b_2}.
}
{\bf Case 2b:} $|A\cap S| \geq 3|S|/4$.\\ 
Let $A^{\ast} \subseteq A \cap S$ be the vertices with no neighbors in $S\setminus A$. Since each vertex has at most one neighbor in $A$, $|A^{\ast}| \geq  |S|/2$. We compute 
\beq{Phi3}{
D(S) \geq \sum_{x \in A^{\ast}} \frac{ e( \{x\} , \overline{S})}{d(x)}= |A^{\ast} | > \frac{|S|}{2}  \quad \text{ and so } \quad  \Phi(S)  \geq \frac{b_1}{2b_2}.
}
 
It follows from \eqref{Phi1}, \eqref{Phi2}, \eqref{Phi3} that  $\Phi(G) \geq \frac{b_1}{8000b_2}$, and so by the Cheeger inequality 
$$\lambda_2 \leq \left(1 - \frac{\Phi^2}{2}\right) \leq  1 - \frac{b_1^2}{2(8000b_2)^2}.$$  
Let $L= 10(8000b_2)^2/b_1^2>-4 /\log\left( 1 - \frac{b_1^2}{2(8000b_2)^2}\right)$. Letting $t=L \log n$ in \eqref{mix},  we see that
$$|P_u^{(t)}(x)- \pi_x| \leq  21\left( 1 - \frac{b_1^2}{2(8000b_2)^2}\right)^{-L\log n} <n^{-3}.$$
\end{proof}

Finally, we give upper bounds on the values $R_v$, the expected number of times that $\cW_v$ returns to $v$ in $T$ steps, where $T$ is as defined in \eqref{T}. Here we refer to the set of vertices within distance $k$ of $v$ as the $k-$neighborhood of a vertex $v$.

\begin{lemma} \label{rv}
 Let $G$ be typical. Let 
\begin{align*}
A&=\set{v:d(v)\leq \frac{np}{100}}.\\
B&=\set{v:N_{10}(v)\cap A=\emptyset}.
\end{align*}
(Thus $A\cap B=\emptyset$.)

Then 
$$1\leq R_v\leq \begin{cases}1+O\bfrac{1}{\log n}&v\in B\text{ or }(v\notin A\cup B\text{ and }N(v)\cap A=\emptyset).\\ 1+\frac{C}{d(v)}& v\in A\\1+\frac{C}{d(v_1)}.&N(v)\cap A=\set{v_1}.\end{cases}$$
where $C$ is an absolute constant. 
\end{lemma}

\begin{proof} 
For the majority case of $v\in B$, we estimate $R_v$ by projecting the random walk in the neighborhood of $v$ onto the nonnegative integers, with $v$ corresponding to zero. Divide the vertices of $V$ into levels based on their distance from $v$. Let $\a$ be an upper bound on the probability the walk moves from level $i\leq 4$ to level $i-1$, $\r\geq \a$ be an upper bound on the probability the walk stays in level $i$, and $\b=1-\a-\r$ (a lower bound on the probability that the walk moves from level $i$ to level $i+1$, that will soon be seen to be non-negative). We couple $\cW_v$ with a random walk $\hW$ on $\set{0,1,2,3,4}$ with parameters $\a,\b,\r$ so that $\hW\leq \cW_v$ whenever $\cW_v$ is within distance 4 of $v$. When $\hW$ is at 4, it moves to 3 at the next step. Otherwise, if $\cW_v$ is at level $0<j\leq 3$ and $\hW$ is at $0<i<j$ then we couple the two walks so that $\hW$ will move left or stay if $\cW_v$ moves left.  Also, if $\cW_v$ is at level $0<j\leq 3$ and $\hW$ is also at $j$ then $\hW$ will move left if $\cW_v$ moves left. If $\hW$ is at 0 and $\cW_v$ is at 1 and $\cW_v$ moves left, then $\hW$ will stay at 0. Here we use the assumption $\r\geq \a$. We will in fact simplify matters by taking $\a=\r$. It remains to define $\a=\r,\b$ and to estimate the expected number of times that $\hW$ visits 0, if it starts there.

Let $E_i$ be an upper bound on the expected number of times that $\hW$, beginning at $i$, visits $0$ in $T$ steps. We have, 
\beq{ineq1}{
E_0\leq 1+\r E_0+(1-\r)E_1\text{ which implies that }E_0 \leq 1+ 2\r+E_1,
}
assuming that $\r\leq 1/2$.

More generally 
\beq{ineq2}{
E_i\leq  \a E_{i-1} + \r E_i +\b E_{i+1}\text{ for }1\leq i\leq 3.
}  
Note also that 
\begin{equation}\label{E4} 
E_3 \leq T\a^2\sum_{\ell\geq0}\brac{\a(1-\a)}^{\ell}=\frac{T\a^2}{1-\a(1-\a)}\leq 2T\a^2
\end{equation} 
since each time the walk moves left from 3 it has a less than $\a^2\sum_{\ell\geq 0}\brac{\a(1-\a)}^{\ell}$ chance of reaching zero before returning to 3. (Note that $\a(1-\a)$  increases for $\a\leq 1/2$ and we will only use this estimate when $\a=o(1)$.) Here $\ell$ is the number of times that $\hW$ moves 2,1,2 before finally moving to zero in two steps. (Included in 2,1,2 there might be some $x,x,\ldots,$ where $x\in\set{1,2}$.)

Summing the inequalities in \eqref{ineq2} for $i=1,2$ yields 
$$ 0 \leq \a E_0 -\b E_1 -\a E_2 +\b E_3\leq  \a E_0 -\b E_1+\b E_3.$$ 
It follows that
$$E_0 \leq 1+ 2\r+E_1 \leq 1 + 2\r+\frac{\a E_0}{\b} +E_3,$$ 
and so 
\begin{equation} \label{E0} 
E_0 \leq \frac{1+2\r+2\a^2 T}{1-\frac{\a}{\b}}. 
\end{equation}
First we assume $np\leq n^{1/100}$. We consider several cases. 

\noindent {\bf Case 1:  $v \in B$.}\\
 Consider $u$ at distance $i \leq 3$ from $v$. Condition \ref{subgraphs} of Lemma \ref{typical} guarantees there are at most two edges from $u$ to level $i-1$ and at most one edge from $u$ to another vertex in level $i$. Since $u$ and all its neighbors have degree at least $np/100$ we can take $\a=\r=\frac{200}{np}$ and $\b=1-2\r$. It follows from \eqref{E0} that for $v \in B$ and $T=L \log n$, 
\beq{Rv1}{
R_v = E_0 \leq \frac{1+O(1/\log n)}{1-O(1/\log n)}= 1+ O\bfrac{1}{\log n}.
}

\noindent {\bf Case 2: $v\in A$.}\\
We observe that if $\cW_v$ is at $w\in N(v)$ then the probability it moves to $v$ in one step can be bounded by 
\beq{xx}{
\frac{1/d(v)}{1/401+1/d(v)}=\frac{1}{1+\frac{d(v)}{401}}.
}
{\bf Explanation:} We have $\sum_{x\in N(w)\setminus v}\psi(w,x)\geq \frac{np/100-1}{4np}\geq 1/401$. This follows from Condition \ref{maxd} of Lemma \ref{typical}. Also, $\psi(w,v)=1/d(v)$.
It follows that
\beq{Rv4}{
R_v \leq 1+\frac{1}{d(v)}\sum_{w\in N(v)}\frac{R_w}{1+\frac{d(v)}{401}} \leq 1+ \frac{1}{d(v)}\sum_{w\in N(v)}\frac{1+o(1)}{1+\frac{d(v)}{401}}\leq 1+\frac{402}{d(v)}.
}
{\bf Explanation:} For the first inequality, we see that when at $v$, $\cW_v$ chooses a neighbor $w$ uniformly from $N(v)$ and then each return to $w$ yields a return to $v$ with probability as given in \eqref{xx}. Each return to $w$ will avoid using the edge $\set{v,w}$. Thus we can invoke Case 1 to justify the second inequality. 

{ \bf Case 3: $v\not \in A \cup B$. } It follows from Lemma \ref{typical} that the 50-neighborhood of $v$ is a tree, and there exists a single vertex in $a \in A$ in the 10-neighborhood of $v$.  Let $N(v)=\set{v_1,v_2,\ldots,v_d}$ where the unique $a\in N_4(v)$ lies in the sub-tree rooted at $v_1$. Let $\alpha$ be the probability the walk moves from $v$ to $v_1$. Let $\r_{\overline{v}}$ be the probability that $\cW_v$ returns to $v$ if edge $(v,v_1)$ is removed. Then we have 
\begin{equation}\label{rvb} 
R_v \leq 1+\alpha R_v+\r_{\overline{v}}R_v.
\end{equation}
{\bf Explanation:} For the first term assume that the walk always returns from $v_1$ and similarly, for the second term, we assume that the walk always returns to $v$, if it returns to $v_j,j\geq 2$.

We compute 
\begin{equation} \label{alpha} 
\alpha = \frac{\frac{1}{ \min \{ d(v), d(v_1)\}}}{\Psi(v)}\leq   \frac{\frac{1}{ \min \{ d(v), d(v_1)\}}}{ (d(v)-1)/d(v)+1/d(v_1)}.
\end{equation} 
Furthermore, if $R_{\overline{v}}$ denotes the expected number of returns to $v$ if edge $(v,v_1)$ is removed then
\beq{rhov}{
1+\r_{\overline{v}}\leq R_{\overline{v}}\leq 1+O\bfrac{1}{\log n}\text{ implying that }\r_{\overline{v}}=O\bfrac{1}{\log n}.
}
where we have used \eqref{Rv1}. 

It follows from \eqref{rvb} that
\beq{ualp}{
R_v\leq \frac{1}{1-\a+O\bfrac{1}{\log n}}.
}

 We now consider two sub cases. \\
{ \bf Case 3a: $v_1 \in A$. }  For $d(v_1)=1$, it follows from \eqref{alpha} that
$\alpha \approx 1/2$. For $d(v_1) \geq 2$, \eqref{alpha} implies $\alpha \leq \frac{1}{d(v_1)}$. Therefore \eqref{ualp} implies
\begin{equation} \label{d1} 
R_v \leq 1 +O\bfrac{1}{d(v_1)}. 
\end{equation}
{ \bf Case 3b: $v_1 \notin A$. }   It follows from \eqref{alpha} that
$\alpha \leq \frac{100}{np} .$ Therefore \eqref{ualp} implies
\begin{equation} \label{d1a} 
R_v \leq  1 +O\bfrac{1}{\log n}.
\end{equation}

Finally we consider $np>n^{1/100}$.
The Chernoff bounds imply that with high probability $d(x)\approx np$ for all $x\in V$. Now,
\beq{npL}{
R_v\leq 1+\frac{T}{\min_{w\in N(v)}d(w)}=1+o(1).
}
This is because there are at most $T$ instances where $\cW_v$ is at a neighbor of $v$ and then the chance of moving to $v$ on the next step is bounded by $\frac{1}{\min_{w\in N(v)}d(w)}$.

The lemma follows from \eqref{Rv1}, \eqref{Rv4}, \eqref{d1}, \eqref{d1a} and \eqref{npL}.
\end{proof}

\begin{remark}
Arguing as for \eqref{E4} we see that if we allow $\hW$ to move to 5 then
\beq{E4a}{
E_4=O\bfrac{1}{(np)^2}.
}
This will be needed later when we discuss the lower bound. In particular, we need it to verify \eqref{Rz}.
\end{remark}

\section{Cover time} \label{the proof}
In this section we prove upper and lower bounds on the cover time (Lemmas \ref{upper} and \ref{lower} respectively), which together imply Theorem \ref{main}. 
\subsection{The upper bound}
In the proof of the upper bound we apply the first visit time lemma. We rely on Lemma \ref{helper} to show Condition ($\ref{rt}$) of the first visit time lemma. 
\begin{lemma} \label{upper} Let $G=G_{n,p},p=\frac{c\log n}{n}$ with $c\geq 1$ and $\om=(c-1)\log n\to\infty$. Then with high probability, $\overline{C}(G) \lesssim n\log n$.
\end{lemma}
\begin{proof} With high probability $G$ is typical, as stated in Lemma \ref{typical}.
Let $u$ be an arbitrary vertex of $G$, let $T_g(u)$ be the time taken to visit all vertices by $\W_u$, and let $U_s$ be the number of vertices that haven't been visited at step $s$. Let $A_s(v)$ be the event that $\W_u$ does not visit $v $ in $[T,s]$. For $t >T$, 
$$\overline{C}(G)= \E{T_g(u)}= \sum_{s>0} \Pr{U_s>0}\leq \sum_{s>0} \min\{ 1, \E{U_s}\}\leq t + \sum_{v \in V} \sum_{s>t} \Pr{ A_s(v)}.$$

Next we apply the first visit time lemma (Lemma \ref{firstvisit}) for $T= L \log n$. Lemma \ref{mixing time claim} guarantees Condition (i). The boundedness of $R_v$, proved in Lemma \ref{rv}, implies the assumptions of Lemma \ref{helper}, thereby guaranteeing Condition (\ref{rt}) holds for $\theta=1/4$ and $K$ sufficiently large. Note that $T\pi_v= \Theta(\log n/n)$, therefore Condition (iii) holds. We apply the lemma and compute 
$$\Pr{ A_t(v)}=\lo e^{-tp_v} + o(e^{-t/KT})\approx e^{-\pi_vt/R_v}. $$

Let $A,B$ be as defined in Lemma \ref{rv}.
Then, by the bounds on $\pi_v$ and $R_v$ given in Claim \ref{piv} and Lemma \ref{rv} respectively, 
\begin{enumerate}
\item  $v\in B$ implies that $\frac{R_v}{\pi_v} \leq (1+\th)\bfrac{1+\ve}{1-\ve}n$ for some $\th=o(1)$.
\item $v\in C=\set{v\notin A\cup B\text{ and }N(v)\cap A=\emptyset}$ also implies that $\frac{R_v}{\pi_v} \leq (1+\th)\bfrac{1+\ve}{1-\ve}n$.
\item $v\in V_k$ (vertices of degree $k$, see Lemma \ref{typical}\eqref{Vk}) implies that $\frac{R_v}{\pi_v} \leq \brac{1+\frac{M}{k}}\bfrac{1+\ve}{1-\ve}n$ for some constant $M>0$.
\end{enumerate}
For ease of notation, let $\tau = (1+\th)\left(\frac{1+\ve}{1-\ve}\right)$ and let $ t= \tau n \log n$.  Lemma \ref{firstvisit} shows that for some $\th_2=o(1)$, 
\begin{align*}
\overline{C} (G)&\leq \tau n \log n + (1+\th_2)\sum_{v \in V} \sum_{s>t} e^{-\pi_vs/R_v}\\
&= \tau n \log n+ (1+\th_2)\sum_{v \in V}\frac{R_v}{\pi_v}e^{-\pi_vt/R_v}\\
&\leq \tau n \log n+ (1+\th_2)|B\cup C| \left(\frac{1+\ve}{1-\ve}\right) + \sum_{k\geq 1}\sum_{v\in V_k\setminus (B\cup C)}\frac{R_v}{\pi_v}e^{-\pi_vt/R_v} \\
&\leq \t n\log n+O(n)+\sum_{k\geq 1}\sum_{v\in V_k\setminus (B\cup C)}\frac{R_v}{\pi_v}e^{-\pi_vt/R_v}.
\end{align*}
We complete the proof of the lemma by showing that
\beq{rest}{
\sum_{k\geq 1}\sum_{v\in V_k\setminus (B\cup C)}\frac{R_v}{\pi_v}e^{-\pi_vt/R_v}=o(n)
}
and then letting $\ve \to 0$.

{\bf Proof of \eqref{rest}:} If $k\leq \La$ then using Condition \ref{Vk} of  Lemma \ref{typical}, we have
\begin{multline}\label{smallk}
\sum_{v\in V_k\setminus (B\cup C)}\frac{R_v}{\pi_v}e^{-\pi_vt/R_v}\leq (3\log n)^{k+1}\times \brac{1+\frac{M}{k}}\bfrac{1+\ve}{1-\ve}n\times \exp\brac{-\frac{\t\log n}{\brac{1+\frac{M}{k}}\bfrac{1+\e}{1-\e}}}\\
=o(n),
\end{multline}
provided $\th\geq 2M/\log\log n$ and $\e$ is sufficiently small.

If $k>\La$ then using Conditions \ref{t1}, \ref{t2} and \ref{set A} of Lemma \ref{typical}, we have
\begin{multline}\label{largek}
\sum_{v\in V_k\setminus (B\cup C)}\frac{R_v}{\pi_v}e^{-\pi_vt/R_v}\leq 3n^{1-\e^2c/4}\times \brac{1+\frac{2C}{\La}}\bfrac{1+\ve}{1-\ve}n\times \exp\brac{-\frac{\t\log n}{\brac{1+\frac{M}{k}}\bfrac{1+\e}{1-\e}}}\\
\leq n^{1-\Omega(1)}.
\end{multline}
Here $2n^{1-\e^2c/4}+n^{17/12-c} \bfrac{np}{100}\leq 3n^{1-\e^2c/4}$ is a bound on $|V\setminus (B\cup C)|$. Note that $A=\emptyset$ with high probability as shown in the proof of Condition \ref{set A}.   Equation \eqref{rest} follows from \eqref{smallk} and \eqref{largek}.
\end{proof}
\subsection{The lower bound}
Finally, we give a lower bound on cover time. We observe that Feige's lower bound \cite{Feige} is only claimed to hold for the simple random walk where each neighbor of the current vertex $v$ is equally likely to be chosen as the next vertex to be visited. The following proof  that the cover time of any reversible random walk is $\Om(n \log n)$,
is due to T. Radzik.

For a walk starting from vertex $u$, the expected first return time $T^+_u$ to $u$ satifies $\E{ T^+_u}=1/\pi(u)$. Also, $\E{ T^+_u}$ is at most the commute time  $K(u,v)$ between $u$ and $v$ ($K(u,v) \ge \E{ T^+_u}$). For at least half the vertices $\pi(u) \le 2/n$. Let $S$ be this set of vertices, all with $K(u,v)\ge \E{ T^+_u }\ge n/2$. Let $K_S=min_{i,j \in S} K(i,j)$ then by \cite{KKLV},
\[
 C(G) \ge  \max_{S \subseteq V} K_S \log |S|\ge (n/4) \log (n/2).
\]
Our results imply that with high probability $C_u(G)\approx n\log n$ for all $u\in V$.
\begin{lemma}  \label{lower} 
Let $G \sim G_{n,p}$ where $p= \frac{c \log n}{n}$ and $\om=(c-1)\log n\to\infty$. Then with high probability, $\overline{C}(G) \gtrsim n\log n$.
\end{lemma}
\begin{proof}
Let $I=[(1-\e)np,(1+\e)np]$. Let $S_0$ be the set of vertices $v$ such that 
\begin{enumerate}[(P1)]
\item $d(v)\in I$.
\item $d(w)\in I$ for $w\in N_2(v)$.
\item $d(w)\leq d(v)$ for $w\in N(v)$. 
\item $v\in B$, where $B$ is defined as in Lemma \ref{rv}.
\end{enumerate}
It follows from Conditions 2,3 and 6 of Lemma \ref{typical} that the number of vertices satisfying P1,P2 and P4 is $n-o(n)$. 
\begin{claim}\label{clS0}
$$|S_0|\geq \frac{n}{5np}.$$
\end{claim}
\begin{proof}
Let $X_0=\set{v:d(v)>(1+\e)np}$ and $Y_0=X_0\cup N(X_0)$. It follows from Conditions 2,3 and 4 of Lemma \ref{typical} that $|Y_0|=o(n)$.

Next let $k_i=(1+\e)np -i$. Then let $X_1=V_{k_1}\setminus Y_0$ and $Y_1=X_1\cup N(X_1)$. In general we let $X_{i+1}=V_{k_{i+1}}\setminus \bigcup_{j\leq i}Y_i$ and $Y_{i+1}=X_{i+1}\cup N(X_{i+1})$. 

We note that $x\in X_i,i\geq 1$ implies that $d(x)=k_i\geq d(w)$ for $w\in N(x)$. This is because $x$ has no neighbors in $V_\ell$ for $\ell<k_i$.

Now Lemma \ref{typical} implies that $\sum_{i=1}^{2\e np}|X_i\cup Y_i|=n-o(n)$. And then our bound on maximum degree of $4np$ implies that
$$|S_0|\geq \sum_{i=1}^{2\e np}|X_i|\geq \frac{n-o(n)}{4np}.$$

\end{proof}
Given the claim, we divide the possible range for $R_v$ into $\log^2n$ sub-intervals and use the pigeon-hole principle to select a subset $S_1\subseteq S_0$ of size $\Omega\bfrac{n}{np\log^2n}$ such that 
\beq{close}{
|R_u-R_v|\leq \frac{1}{\log^2n}\text{ for }u,v\in S_1.
}

Next let $S$ be a maximum size subset of vertices of $S_1$ such that no two vertices of $S$ are within distance 10 of each other. We show that with $\e$ sufficiently small and for $\d=3\e$, that with high probability the set $S$ will not be covered at time $t =(1- \d) n \log n $. 

We show next that a greedy algorithm applied to $S_1$ produces a set $S$ of size at least $\frac{n-o(n)}{ (4np)^{10}}$. After selecting $k$ vertices from $S_1$, there will be at least $|S_1|- k (4np)^{10} $ vertices in $S_0$ available for the next choice of vertex for $S$. Therefore 
$$|S| \geq \frac{|S_1|}{(4np)^{10}}\geq \frac{n}{(5np)^{11}}.$$

Let $S(t)$ denote the number of vertices in $S$ that have not been visited by the random walk at step $t$. For $t >T$, 
$$ \E{ S(t)}\geq -T+\sum_{v \in S} \Pr{A_v(t)}.$$
We compute 
$$\Pr{ A_v(t)}\approx\frac{1}{(1+ \frac{\pi_v}{R_v})^t}\gtrsim\exp\left(-(1-o(1)) \frac{t(1+\e)}{(1-\e)n}\right)\geq n^{-(1-\e/2)}.$$
It follows 
$$\E{S(t)}=\Omega\bfrac{n^{\e/2}}{(5np)^{11}}\to\infty,$$
assuming that 
\beq{largec}{
np\leq n^{\e/25}.
}
We make this assumption for now and deal with $np>n^{\e/25}$ in Section \ref{highd}.

As in earlier papers, we apply the Chebyshev inequality to show that $S(t)\neq\emptyset$ with high probability. To estimate $\E{S(t)(S(t)-1)}$, we estimate the probability that two distinct vertices $ u, v \in S$ have not been visited by time $t$. Let $\Gamma$ be obtained from $G$ by contracting $u$ and $v$ into a single vertex, which we call $z$. 
\begin{claim}\label{clpW}
The probability a random walk $\overline{\W}_w,w\neq u,v$ in $\Gamma $ doesn't visit $z$ in $t$ steps equals the probability that the random walk $\W_w$ in $G$ visits neither $u$ nor $v$ in $t$ steps. 
\end{claim}
\begin{proof}
Let $\om=(w=v_0,v_1,\ldots,v_t)$ be a walk that does not visit $u,v$. Let $p_W$ and $\hp_W$ be the probabilities that $\cW_u$ follows $W$ in $\om$ and in $\G$ respectively. Then
\beq{pW}{
p_W=\prod_{i=0}^{t-1}\frac{\psi(v_i,v_{i+1})}{\Psi(v_i)},
}
The claim follows from \eqref{pW} and the fact that if $\hps$ equals the induced values of $\psi$ in $\G$ then
\beq{pW2}{
\hps(x,y)=\psi(x,y)\text{ for all }x,y\notin N(u)\cup N(v).
}
and
\beq{pW3}{
\hps(x,z)=\psi(x,u)\text{ for all }x\in N(u).
}
Equation \eqref{pW2} is clear and equation \eqref{pW3} follows from our Condition P3.
\end{proof}
Note that \eqref{E4a} implies that the expected number of returns to $v$ after reaching distance $4$ from $v$ is $O(1/\log^2n)$. Therefore, since all paths between $u$ to $v$ contain vertices at distance at least four from $u$ and $v$, 
\beq{Rz}{
R_z= \frac{1}{2} R_u +  \frac{1}{2} R_v+O\bfrac{1}{\log^2n}.
} 
With respect to steady state probabilities, it follows from \eqref{pW2}, \eqref{pW3} that we have
\beq{piz}{
\pi_z=\p_u+\p_v.
}
It is straightforward to check that the conditions of Lemma \ref{firstvisit} hold for $\G$ with $T=O(\log n)$. It follows from \eqref{PrAt}, \eqref{Rz} and \eqref{piz} that
\begin{align}
\Pr{A_z(t)} &\approx \exp\brac{-\frac{(\p_u+\p_v)t}{\frac{1}{2} R_u +  \frac{1}{2} R_v+O\bfrac{1}{\log^2n}}}\label{newPA}\\
&=\exp\brac{-\frac{(\p_u+\p_v)t}{R_u +O\bfrac{1}{\log^2n}}}\nonumber\\
&\approx \exp\brac{-\frac{\p_ut}{R_u}}\times \exp\brac{-\frac{\p_vt}{R_v}}\nonumber\\
&\approx \Pr{A_u(t)}\times \Pr{A_v(t)}.\nonumber
\end{align}
It follows that
$$ \E{S(t)(S(t)-1)}\lesssim \E{S(t)}^2$$
and so 
$$ \Pr{ S(t) >0} \geq \frac{\E{S(t)}^2}{\E{S(t)^2}} = \frac{\E{S(t)^2}}{\E{S(t)(S(t)-1)}+\E{S(t)}}\geq \frac{1}{1+o(1)+\E{S(t)}^{-1}}=1-o(1).$$
\end{proof}
\subsection{High average degree case}\label{highd} 
We show how to amend the above argument for the case where $np\geq n^{\e/25}$. All vertices satisfy P1,P2 and P4 and we drop P3. We can however claim that with high probability 
\beq{degree}{
|d(v)-np|\leq \sqrt{10np\log n}\text{ for all }v\in V.
}
This follows from applying the Chernoff bounds to $d(v)\sim Bin(n-1,p)$.

We cannot claim \eqref{pW3}, but because $d(z)\approx 2d(u)$, we have instead that with high probability
\beq{pW10}{
\hps(x,z)=\begin{cases}\psi(x,u)&d(x)\leq d(u).\\\psi(x,u)-\frac{1}{d(u)}+\frac{1}{d(x)}&d(x)>d(u).\end{cases}
}
Now, \eqref{degree} implies that
$$\frac{1}{d(u)}-\frac{1}{d(x)}=O\bfrac{\log^{1/2}n}{(np)^{3/2}}.$$
It follows from this that instead of \eqref{piz} we have
$$\p_z=(\p_u+\p_v)\brac{1+O\bfrac{\log^{1/2}n}{(np)^{1/2}}}.$$
Going back to \eqref{newPA} we obtain
\begin{align*}
\Pr{A_z(t)} &\approx \exp\brac{-\frac{(\p_u+\p_v)\brac{1+O\bfrac{\log^{1/2}n}{(np)^{1/2}}}t}{\frac{1}{2} R_u +  \frac{1}{2} R_v+O\bfrac{1}{\log^2n}}}\\
&\approx\exp\brac{-\frac{(\p_u+\p_v)t}{R_u +O\bfrac{1}{\log^2n}}}
\end{align*}
and the proof continues as for the previous case.
\section{conclusion}
We have given an asymptotically tight analysis of the cover time of a biased random walk on $G_{n,p}$. It would certainly be of interest to consider other possible biased walks and also to consider the analysis of the walk in this paper on other models of a random graph. In particular, it would be of interest to analyze the performance of this walk on a preferential attachment graph.

\end{document}